\documentclass[11pt,letterpaper]{article}
\usepackage[latin1]{inputenc}
\usepackage{amsmath}
\usepackage{amsfonts}
\usepackage{amsthm}
\usepackage{amssymb}
\usepackage{graphicx}
\usepackage{xcolor}
\usepackage[titletoc,title]{appendix}
\newtheorem{thm}{Theorem}[section]
\newtheorem{lemma}[thm]{Lemma}
\newtheorem{coro}{Corollary}[thm]
\newtheorem{defn}[thm]{Definition}

\author{William M. Brummond\footnote{University of Missouri -  St. Louis.  Email: wmbrummond@umsl.edu.}}
\date{June 5, 2019}
\title{Kirkman Systems that Attain the Upper Bound on the Minimum Block Sum, for Access Balancing in Distributed Storage}

\begin{document}

\maketitle

\begin{abstract}
We study a class of combinatorial designs called Kirkman systems, and we show that infinitely many Kirkman systems are well-distributed in a precise sense.  Steiner triple systems of order $n$ can achieve a minimum block sum of $n$.  Kirkman triple systems form parallel classes from the blocks of Steiner triple systems.  We prove that there are an infinite number of Kirkman triple systems that have a minimum block sum of $n$.  We expand this to quadruple systems.  
These concepts can then be applied to distributed storage to spread data across the servers, and servers across locations,  using Kirkman triple systems, while having data well distributed by popularity, measured by the minimum block sum.
\end{abstract}

\section{Introduction}

Due to their unique combinatorial design, Steiner systems have many applications in coding theory.  One application is related to distributed storage and information retrieval \cite{Dau_Milenkovic}.  A very large database is often stored on several servers, with no one server containing all the data.  Different chunks of data have different levels of popularity in being accessed.  The ideal situation is to spread the chunks of data over several servers so that each server is accessed the same number of times.  This will result in more efficient accessing of the data.  

Dau and Milenkovic \cite{Dau_Milenkovic} use Steiner triple systems to achieve this by meeting both of the following requirements:
\begin{itemize}
\item 
	Data is distributed across servers based on popularity.

	Each chunk of data is assigned a value indicating its popularity.  If there are $n$ chunks of data, each chuck is assigned a unique value from $\{0, 1, ..., n-1\}$, where smaller values represent more popular data and larger values represent less popular data.  Chunks of data are distributed across the servers so that the sum of the values on any server (``min-sum") is never less than $n$.
\item 
	The number of instances that the same chunks of data are on the same server is limited.

	Using Steiner triple systems, any two chunks of data are on only one server.
\end{itemize}

For example, suppose a database has 9 chunks of data, with popularity of access of each chunk measured by the integers $\{ 0, 1,..., 8 \}$, with the smallest number representing the most popular and the largest number representing the least popular.  We might want to distribute this data over 12 servers such that:
\begin{itemize}
\item 
	each server has three chunks of data,
\item 
	the sum of the popularity on any server (``min-sum") is never less than 9, and
\item 
	any two chunks of data are on exactly one server.
\end{itemize}  
This can be accomplished, using a Steiner triple system, as follows:
\begin{center}
\begin{tabular}{| c | c | c|}
\hline
Server & Chunks of Data & Sum of Popularity\\
\hline
A & 0, 1, 8 & 9\\
B & 0, 2, 7 & 9\\
C & 0, 3, 6 & 9\\
D & 0, 4, 5 & 9\\
E & 1, 2, 6 & 9\\
F & 1, 3, 5 & 9\\
G & 1, 4, 7 & 12\\
H & 2, 3, 4 & 9\\
I & 2, 5, 8 & 15\\
J & 3, 7, 8 & 18\\
K & 4, 6, 8 & 18\\
L & 5, 6, 7 & 18\\
\hline
\end{tabular}
\end{center}
This paper takes this concept one step further.

The servers could be located in different geographical locations, so as to reduce the risk that all the servers could be destroyed in one catastrophic event.  The different locations could be different buildings, different cities, different countries or any other separate locations.  We will group the servers in different locations, so that if transmissions between locations were to be disrupted, each location would continue to have access to all of the data.  Using the above example, we would want four locations, each containing three servers, which when combined contain all the chunks of data, but continue to have the above restrictions on the popularity of data on any one server.  This grouping can be accomplished, using a Kirkman triple system, by grouping the servers as follows:
\begin{center}
\begin{tabular}{| c | c |}
\hline
Location & Servers \\
\hline
I & A, H, L\\
II & B, F, K \\
III & C, G, I\\
IV & D, E, J\\
\hline
\end{tabular}
\end{center}

This results in each location having access to all the data, while continuing to meet the min-sum requirement and to have no two chunks of data on more than one server.

Note that a Kirkman triple system can be used to address how to group chunks of data by location so that all the data is in one location, but Kirkman systems have not been used to address popularity.  We use Kirkman systems to address distribution by location, while reflecting levels of popularity.  

This paper uses combinatorial design theory to prove both of the following:
\begin{enumerate}
\item 
	There are an infinite number of Kirkman triple systems that have min-sums that reach their upper bounds.
\item 
	There are an infinite number of Kirkman quadruple systems that have min-sums that reach their upper bounds.
\end{enumerate}

These conclusions can then be used to distribute servers by location so that each location has all of the data, while the chunks of data have been distributed across the servers by popularity.

\section{Preliminaries}

This paper uses notation consistent with that of Dau and Milenkovic.

\subsection{Steiner systems}

A Steiner triple system is a pair $(\mathcal{S}, \mathcal{B})$, where $\mathcal{S}$ is a set of $n$ elements and $\mathcal{B}$ is a set of 3-subsets of $\mathcal{S}$, called blocks, with every two elements of $\mathcal{S}$ being contained in exactly one $B \in \mathcal{B}$.  Such a system will be referred to as $STS(n)$.  

In general, a Steiner system $(\mathcal{S}, \mathcal{B})$  will be denoted as $S(t, k, n)$, where $| \mathcal{S} | = n$, each block is of size $k$ and every $t$ elements of $\mathcal{S}$ are contained in exactly one $B \in \mathcal{B}$.
\begin{defn}
	[parallel class]
	A parallel class in $(\mathcal{S}, \mathcal{B})$ is a subset of $\mathcal{B}$ that partitions $\mathcal{S}$.
\end{defn}

\begin{defn}
	[resolvable]
	A Steiner system $(\mathcal{S}, \mathcal{B})$ is resolvable if the blocks of $\mathcal{B}$ can be partitioned into parallel classes.
\end{defn}

\begin{defn}
	[Kirkman triple systems]
	A resolvable Steiner triple system of order $n$ is known as a Kirkman triple system, and is denoted $KTS(n)$. 
\end{defn}

\begin{defn}[min-sum of design $\mathcal{B}$]\cite[page 1647]{Dau_Milenkovic}
	The min-sum of Steiner system $(\mathcal{S},\mathcal{B})$ for $\mathcal{S} = \{0, 1, ..., n-1 \}$ is
	$$min_\Sigma(\mathcal{B}) := \min_{B \in \mathcal{B}}sum(B),$$
	where 
	$$sum(B) = \sum_{x \in B} x. $$
In situations where $\mathcal{B}$ is obvious, we will sometimes write this as $min_\Sigma$ for convenience.
\end{defn}

Dau and Milenkovic \cite{Dau_Milenkovic} addressed the min-sum of Steiner triple systems in 2017, based in part on Bose and Skolem constructions.  To date, there has been no expansion of the min-sum to Kirkman triple systems.  This paper applies the min-sum concept to Kirkman triple systems and Kirkman quadruple systems.

\subsection{Kirkman triple systems}

We first consider the possible values of $n$ for Kirkman triple systems of order $n$.

For $STS(n)$ with design $(\mathcal{S}, \mathcal{B})$, the following are well known \cite[Theorem 1.1.3]{lindner}
\begin{align*}
n &\equiv 1 \text{ or } 3 \pmod{6},\\
| \mathcal{B} | & = \frac{n(n-1)}{6}.
\end{align*}

Each $x \in \mathcal{S}$ must be contained in $r = \frac{n-1}{2}$ blocks of $\mathcal{B}$.  
The size of each parallel class, $\pi$, must be 
$$| \pi | = \frac{n}{3}.$$

Therefore, since $2 \mid (n-1)$ and $3 \mid n$, to be a Kirkman triple system,
$$n \equiv 3 \pmod{6}.$$

Kirkman triple systems deal with blocks of length 3.  The following section examines blocks of length 4.

\subsection{Kirkman quadruple systems}

This subsection examines $S(3,4,n$) of design $(\mathcal{S}, \mathcal{B})$, where:
\begin{itemize}
	\item $| B | = 4$, for all $B \in \mathcal{B}$;
	\item Every three elements of $\mathcal{S}$ is contained in exactly one $B \in \mathcal{B}$; and
	\item $(\mathcal{S}, \mathcal{B})$ is resolvable; i.e,  the blocks, $B$ of $\mathcal{B}$, can be partitioned into parallel classes.
\end{itemize}

For $S(3,4,n)$, the following are well known \cite[page 146]{lindner}
\begin{align*}
n &\equiv 2 \text{ or } 4 \pmod{6},\\
| \mathcal{B} | & = \frac{n(n-1)(n-2)}{24}.
\end{align*}

The size of each parallel class, $\pi$, must be 
$$| \pi | = \frac{n}{4},$$
since each of the $n$ elements must be in exactly one $B \in \pi$.
Therefore, to be a Kirkman quadruple system (``$KQS$"),
$$n \equiv 4 \text{ or } 8 \pmod{12}.$$
The number of parallel classes is 
$$\frac{|\mathcal{B}|}{| \pi |} = \frac{(n-1)(n-2)}{6}.$$

\section{Kirkman systems that reach the upper bound for the min-sum}

\subsection{The upper bound on min-sum}

Dau and Milenkovic \cite[page 1647]{Dau_Milenkovic} proved that the upper bound on the min-sum for any Steiner system $S(t, k, n)$ is
$$min_\Sigma \leq \frac{n(k-t+1) +k(t-2)}{2}.$$
Therefore, the upper bounds on the min-sum for $STS(n)$ and $SQS(n)$ are $n$ and $n+2$, respectively\footnote{For $STS(n)$, $k=3$, $t=2$, 
$$min_\Sigma \leq \frac{n(3-2+1) +3(2-2)}{2}=n.$$
For $SQS(n)$, $k=4$, $t=3$, 
$$min_\Sigma \leq \frac{n(4-3+1) +4(3-2)}{2}=n+2.$$}.

In this section, we prove that there are an infinite number of Kirkman triple systems and Kirkman quadruple systems, constructed from designs of a smaller order, that reach the upper bound for the min-sum.  

\subsection{Kirkman triple systems that reach the upper bound for the min-sum}

The following will show that there are an infinite number of Kirkman triple systems that reach the upper bound on the min-sum.   

\begin{thm}\label{theorem_KTS(3n)}
Let $(\mathcal{S}, \mathcal{B})$ be a $KTS(n)$ with $min_\Sigma(\mathcal{B}) = n$.  Then there exists a $KTS(3n)$, $(\mathcal{S'}, \mathcal{B'})$, with $min_\Sigma(\mathcal{B'}) = 3n$.
\end{thm}
\begin{proof}
Let $\pi_1, \pi_2,...,\pi_\frac{n-1}{2}$ be the parallel classes of $(\mathcal{S}, \mathcal{B})$, where $\mathcal{S} = \{0, 1, ...,n-1 \}$.  Let $(a_{i,j}, b_{i,j}, c_{i,j})$ be the blocks of $\pi_i$ for $1 \leq i \leq \frac{n-1}{2}$, and $1 \leq j \leq \frac{n}{3}$, arranged so that $a_{i,j} < b_{i,j} < c_{i,j}$, and
$$a_{i,j} + b_{i,j} + c_{i,j} \geq n.$$
Define the parallel classes of $(\mathcal{S'}, \mathcal{B'})$ as follows
\begin{itemize}
	\item $\pi'_0 = \{ (t, t+n, t+2n) \mid 0 \leq t \leq n-1 \}$.
	\item $\pi'_{1,i} = \{ (a_{i,j}, b_{i,j}, c_{i,j}+2n), (a_{i,j}+n, b_{i,j}+n, c_{i,j}), (a_{i,j}+2n, b_{i,j}+2n, c_{i,j}+n) \mid  1 \leq j \leq \frac{n}{3} \}$ for $1 \leq i \leq \frac{n-1}{2}$.
	\item $\pi'_{2,i} = \{ (a_{i,j}, c_{i,j}, b_{i,j}+2n), (a_{i,j}+n, c_{i,j}+n, b_{i,j}), (a_{i,j}+2n, c_{i,j}+2n, b_{i,j}+n) \mid  1 \leq j \leq \frac{n}{3} \}$ for $1 \leq i \leq \frac{n-1}{2}$.
	\item $\pi'_{3,i} = \{ (b_{i,j}, c_{i,j}, a_{i,j}+2n), (b_{i,j}+n, c_{i,j}+n, a_{i,j}), (b_{i,j}+2n, c_{i,j}+2n, a_{i,j}+n) \mid  1 \leq j \leq \frac{n}{3} \}$ for $1 \leq i \leq \frac{n-1}{2}$.
\end{itemize}

Items 1 and 2, below, show that $(\mathcal{S'}, \mathcal{B'})$ is $STS(3n)$.  Item 3 shows that is resolvable.  Item 4 proves that $min_\Sigma(\mathcal{B'})=3n$.
\begin{enumerate}
	\item Prove:  Each block is unique.\\
	\\
		Since each block, $(a_{i,j}, b_{i,j}, c_{i,j})$, is Steiner, any two elements determine the third, and each of three are distinct.
		
		Look at the following three groups of coordinates:
		\begin{itemize}
		\item The coordinates $a_{i,j}$, $b_{i,j}$, and $c_{i,j}$ are in the range [$0, n-1$].  
		\item When we add $n$ to a coordinate, it is in the range [$n, 2n-1$].
		\item When we add $2n$ to a coordinate, it is in the range [$2n, 3n-1$].  
		\end{itemize}
		
		So for  $\pi'_{1,i}$, $\pi'_{2,i}$ and $\pi'_{3,i}$, any two blocks will have the same coordinates only if the same $a_{i,j}$, $b_{i,j}$, and $c_{i,j}$ are included in both, with the same additions of 0, $n$, or $2n$.  But, it is obvious that this never occurs.  Therefore, each block is unique. 
		
	\item Prove:  The total number of blocks in $\mathcal{B'}$ is $\frac{3n(3n-1)}{6}$.
		\begin{align*}
		\text{Total number of blocks } &=
		|\pi'_0|+ \sum_{k=1}^3 \sum_{i=1}^{\frac{n-1}{2}} |\pi'_{k,i}|\\
		&= n + 3 \cdot \frac{n-1}{2} \cdot n \\
		&= \frac{3n(3n-1)}{6}.
		\end{align*}
		
	\item Prove:  Each parallel class includes all the elements of $\mathcal{S'} = \{0, 1, ..., 3n-1 \}$.
	
		For $\pi'_0$, it is obvious that all the elements of $\mathcal{S'}$ are included. 

		Each parallel class in $\pi'_{1,i}$, $\pi'_{2,i}$ and $\pi'_{3,i}$ is created by taking a parallel class from $KTS(n)$ and adding to each coordinate of that parallel class 0, $n$ and $2n$. Since each original parallel class from $KTS(n)$ contains the elements of $\mathcal{S}$, then each new parallel class includes all the elements of $\mathcal{S'} = \{0, 1, ..., 3n-1 \}$.
	\item Prove:  $min_\Sigma(\mathcal{B'}) = 3n$.
	
		We need to show that sum of the coordinates of each block is greater than or equal to $3n$.
		
		For $\pi'_0$, 
		$$t + t + n + t + 2n = 3t + 3n \geq 3n.$$
		
		The parallel classes $\pi'_{1,i}$, $\pi'_{2,i}$ and $\pi'_{3,i}$ are created by adding at least $2n$ to each block of $\mathcal{B}$.  Since each block of $\mathcal{B}$ has a sum of at least $n$, the sum of each block from $\pi'_1$, $\pi'_2$ and $\pi'_3$ is at least $3n$.
\end{enumerate}
Therefore, $(\mathcal{S'}, \mathcal{B'})$ is a $KTS(3n)$ with $min_\Sigma(\mathcal{B'}) = 3n$.
\end{proof}
\begin{coro}\label{induction_kts}
For every $k \geq 1$, there exists a $KTS(n= 3^k)$ with the maximum min-sum.
\end{coro}
\begin{proof}
This is a proof by mathematical induction.\\

\underline{Base case}.  For $k=1$, we have $n= 3$.  Then
\begin{align*}
\mathcal{S}_1 &= \{0, 1, 2\}\\
\mathcal{B}_1 &= \{ (0, 1, 2) \}\\
\pi_1 &= \{ (0, 1, 2) \}\\
min_\Sigma(\mathcal{B}_1) &= 3 = n.
\end{align*}
Therefore, there exists $KTS(n)$ with $min_\Sigma(\mathcal{B}_1)=n$ for $k=1$.

\underline{Inductive Case}.  Assume that there exists $k \geq 1$, such that for $n=3^k$, there exists $(\mathcal{S}_k, \mathcal{B}_k)$, a $KTS(n)$ where $\mathcal{S}_k = \{0, 1, ..., n-1 \}$ with $min_\Sigma(\mathcal{B}_k) = n$.  Then, based on Theorem \ref{theorem_KTS(3n)},  there exists $(\mathcal{S}_{k+1}, \mathcal{B}_{k+1})$, a $KTS(3n)$ where $\mathcal{S}_{k+1} = \{0, 1, ..., 3n-1 \}$ with $min_\Sigma(\mathcal{B}_{k+1}) = 3n$. 

Therefore, there exists $KTS(n = 3^k)$ for all $k \geq 1$, with the maximum min-sum.
\end{proof}

\subsection{Kirkman quadruple systems that reach the upper bound for the min-sum}

The following looks at $KQS(n)$ from a graphical standpoint.  

\begin{defn}
[KQS(n) graph]  A $KQS(n)$ graph is a  regular graph (i.e., each vertex has the same degree), with $n$ vertices and $\frac{n(n-1)(n-2)}{6}$ edges, which contains $\frac{n(n-1)(n-2)}{24}$ 4-cycles, which can be grouped into $\frac{(n-1)(n-2)}{6}$ parallel classes.  Each 3-cycle is included in exactly one of the 4-cycles.
\end{defn}

Let $G$ be a $KQS(n)$ graph with vertex labels $S = \{ 0, 1, ..., n-1\}$.  Each of the 4-cycles are the blocks of the $KQS$.  We know that each pair of vertices $(a,b)$, where $a,b \in S$, appears in $\frac{n-2}{2}$ 4-cycles.  Figure \ref{fig:4-cycles} illustrates two graphs each showing two blocks of $KTS(8)$, with the two blocks representing a parallel class.  The blocks on the left, $(0, 2, 5, 7)$ and $(1, 3, 4, 6)$, and on the right, $(0, 5, 2, 7)$ and $(1, 3, 6, 4)$, are equivalent for our purposes since we are not concerned with the order of the coordinates of the quadruple.
\begin{figure}[h]
	\includegraphics[scale=0.35]{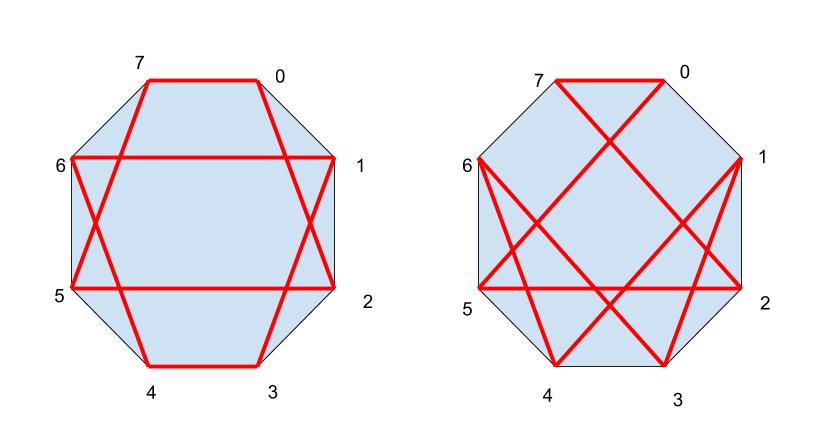}
	\centering
	\caption{4-cycles that are equivalent for our purposes.}
	\label{fig:4-cycles}
\end{figure}\\
For $SQS(n)$,
\begin{align*}
\text{number of blocks} &= \frac{n(n-1)(n-2)}{24},\text{ and}\\
\text{total number of edges} &= 4\frac{n(n-1)(n-2)}{24} = \frac{n(n-1)(n-2)}{6},
\end{align*}
since each block has 4 edges.  The total number of unique edges is $\frac{n(n-1)}{2}$.

For $SQS(2n)$,
\begin{align*}
\text{number of blocks} &= \frac{2n(2n-1)(2n-2)}{24} = \frac{n(2n-1)(2n-2)}{12},\text{ and}\\
\text{total number of edges} &= 4\frac{2n(2n-1)(2n-2)}{24} = \frac{n(2n-1)(2n-2)}{3}.
\end{align*}

The following shows that a $KQS(2n)$ with $min_\Sigma = 2n+2$ can be constructed from two $KQS(n)$ with $min_\Sigma = n+2$.  First, we need to show that for 
\begin{itemize}
	\item $S = \{ 0, 1, ..., n-1\}$, where $4 \mid n$, and
	\item $T = \{ (a,b) \mid a < b ; \; a, b \in S \}$,
\end{itemize}  
$T$ can be partitioned into $\frac{n}{2}$-subsets:  $p_1, p_2,..., p_{n-1}$, such that the pairs of each $p_i$ partition $S$ for for each $i \in \{1, 2, ..., n-1 \}$.\\
\\
First, Lemma \ref{partition_2} will prove this where $2 \mid n$, $4 \nmid n$.  Then, Corollary \ref{partition_2_4} proves this for $2 \mid n$.
\begin{lemma}\label{partition_2}
\emph{(Partition of $\{ (a,b) \mid a < b ; \; a, b \in S \}$)} \\ 
Let 
\begin{itemize}
	\item $S = \{ 0, 1, ..., n-1\}$, where $2 \mid n$, $4 \nmid n$,
	\item $T = \{ (a,b) \mid a < b ; \; a, b \in S \}$.
\end{itemize}  
$T$ can be partitioned into $\frac{n}{2}$-subsets:  $p_1, p_2,..., p_{n-1}$, called parallel classes, such that the pairs of each $p_i$ partition $S$ for each  $i \in \{1, 2, ..., n-1 \}$.
\end{lemma}
\begin{proof}
For the trivial case of $n=2$, 
\begin{align*}
S & = \{0, 1\},\\
T &= \{ (0,1) \},\\
p_1 &= \{ (0,1) \}.
\end{align*}
The following addresses $n > 2$.  Let 
\begin{align*}
p_{Ei} &= \bigg\{ \left( i, i+\frac{n}{2} \right), (i+t,i-t) \pmod{n} \mid t =1, 2, ..., \frac{n}{2}-1 \bigg\} \\
&\qquad \text{ for } i \in \bigg\{0, 1, ...,  \frac{n}{2}-1 \bigg\},\\
p_{Oi} &= \bigg\{ (2t,2t + i) \pmod{n} \mid t =1, 2, ..., \frac{n}{2} \bigg\} \text{ for } i \in \bigg\{ 1, 3, ..., \frac{n}{2} - 2 \bigg\}.\\
p_{Oi+1} &= \bigg\{ (2t+1,2t+1 + i) \pmod{n} \mid t =1, 2, ..., \frac{n}{2} \bigg\} \text{ for } i \in \bigg\{ 1, 3, ..., \frac{n}{2} - 2 \bigg\}.
\end{align*}
Graphically, for $n=10$, $p_{Ei}$ looks like the pairs of vertices connected by the edges on the graphs shown in Figure \ref{fig:pEi},\\
\begin{figure}[h]
	\includegraphics[scale=0.5]{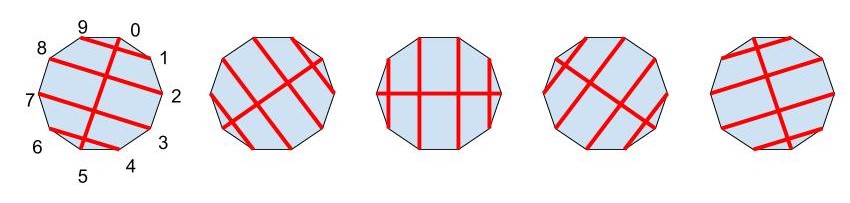}
	\centering
	\caption{$p_{Ei}$ for $n=10$.}
	\label{fig:pEi}
\end{figure}
\\$p_{Oi}$ and $p_{Oi+1}$, for $i=1$, look like the pairs of vertices connected by the edges on the graphs of Figure \ref{fig:pO1},
\begin{figure}[h]
	\includegraphics[scale=0.4]{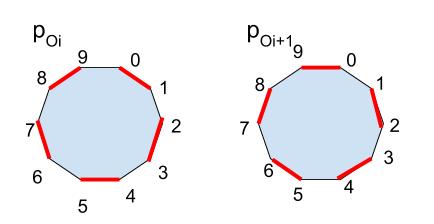}
	\centering
	\caption{$p_{Oi}$ and $p_{Oi+1}$, for $i=1$.}
	\label{fig:pO1}
\end{figure}\\
and $p_{Oi}$ and $p_{Oi+1}$, for $i=3$, look like the pairs of vertices shown in Figure \ref{fig:pO3}.\\
\begin{figure}[h]
	\includegraphics[scale=0.4]{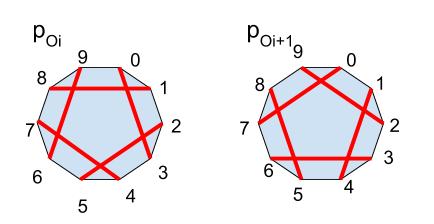}.
	\centering
	\caption{$p_{Oi}$ and $p_{Oi+1}$, for $i=3$.}
	\label{fig:pO3}
\end{figure}\\
Then, let
\begin{align*}
P_E &= \bigg\{ p_{Ei} \mid  i =1, 2,...,  \frac{n}{2}-1  \bigg\},\\
P_O &= \bigg\{ p_{Oi}, \, p_{Oi+1} \mid  i =1, 3, ..., \frac{n}{2} - 2 \bigg\},
\end{align*}
Since $2 \mid n$ and $4 \nmid n$, then
\begin{itemize}
\item for each pair $(x, y) \in p_{Ei}$, the difference between the two elements in the pair is either even or  $\frac{n}{2}$, and
\item for each pair $(2t, 2t+i) \in p_{Oi}$, and each pair $(2t+1, 2t+1+i) \in p_{Oi+1}$, the difference between the two elements in the pair is always i, an odd number less than $\frac{n}{2}$.
\end{itemize}     
Therefore, there are no pairs in a $p_i \in P_E$ that are in any $p_j \in P_O$, and vice versa.

Next, we will show that any pair appears exactly once in any $p_{Ei} \in P_E$.  The order of the elements in pair do not matter.
\begin{itemize}
\item Suppose there exists $(i,i+\frac{n}{2}) = (j,j+\frac{n}{2})$ where $i \neq j$.  Then, $i \equiv  j+\frac{n}{2} \pmod{n}$.  However, there are no $p_i, p_j \in P_E$ where $j-i \equiv \frac{n}{2} \pmod{n}$, because $\mid j-i \mid < \frac{n}{2}$.  Therefore, this is a contradiction.
\item Suppose there exists $(i+t,i-t) = (j+s,j-s)$ where $i \neq j$.  Then, there are two possible cases:
	\begin{itemize}
	\item Case 1. $i+t \equiv j+s \pmod{n}$, and $i-t \equiv j-s \pmod{n}$.\\
	Then, from the first equation we have $i-j+t \equiv s$.  Substituting this into the second equation we get
	\begin{align*}
	i-t &\equiv j -i+j-t \pmod{n}\\
	2i &\equiv 2j,
	\end{align*}
	which is not possible because  $ i \neq j \in \{0, 1, ..., \frac{n}{2}-1 \}$.
	\item Case 2. $i+t \equiv j-s \pmod{n}$, and $i-t \equiv j+s \pmod{n}$.\\
	Then, from the first equation we have $i-j+t \equiv -s$.  Substituting this into the second equation we get
	\begin{align*}
	i-t &\equiv j -i+j-t\\
	2i &\equiv 2j,
	\end{align*}
	which is not possible because $ i \neq j \in \{0, 1, ..., \frac{n}{2}-1 \}$.
	\end{itemize}
\item Suppose there exists $\left( i, i+\frac{n}{2} \right) = (j+t,j-t)$ where $i \neq j$.  Then, there are two possible cases:
	\begin{itemize}
	\item Case 1. $i \equiv j+t \pmod{n}$, and $i+\frac{n}{2} \equiv j-t \pmod{n}$.\\
		Substituting the first equation into the second equation we get
		\begin{align*}
		j+t+\frac{n}{2} &\equiv j-t \pmod{n}\\
		2t +\frac{n}{2}&\equiv 0\\
		t &\equiv \frac{n}{4} \text{  or  } \frac{3n}{4},
		\end{align*}
		which is not possible because $ t \in \{1, 2, ..., \frac{n}{2}-1 \}$ and $4 \nmid n$.
	\item Case 2. $i \equiv j-t \pmod{n}$, and $i+\frac{n}{2} \equiv j+t \pmod{n}$.\\
		Substituting the first equation into the second equation we get
		\begin{align*}
		j-t+\frac{n}{2} &\equiv j+t \pmod{n}\\
		\frac{n}{2}&\equiv 2t\\
		t &\equiv \frac{n}{4} \text{  or  } \frac{3n}{4},
		\end{align*}
		which is not possible because $ t \in \{1, 2, ..., \frac{n}{2}-1 \}$ and $4 \nmid n$.
	\end{itemize}
\end{itemize}

Therefore, any pair appears exactly once in any $p_{Ei}$.

Next, we will show that any pair appears exactly once in any $p_{Oi}$.  Suppose, $(2t_1, 2t_1 + i) = (2t_2, 2t_2 +i)$, where $t_1 \neq t_2$.  Since, $t_1, t_2 \in \{ 1, 2, ..., \frac{n}{2} \}$, the only possibility is that $2t_1 \equiv 2t_2+i$ and $2t_1 +i \equiv 2t_2$.  Substituting the first equation into the second
\begin{align*}
2t_2+i +i &\equiv 2t_2 \pmod{n}\\
2i &\equiv 0.
\end{align*} 
But this is not possible because $i \in \{ 1, 3, ..., \frac{n}{2} - 2 \}$.

Similarly, any pair appears exactly once in any $p_{Oi+1}$.

Next, we will show that any pair in $p_{Oi}$ does not appear in $p_{Oi+1}$.  Suppose $(2t_1, 2t_1 +i) = (2t_2 +1 , 2t_2 +1 +j)$.  The only possibility is for $ 2t_1 \equiv 2t_2 +1 +j$ and $2t_1 +i \equiv 2t_2 +1$.  Substituting the first equation into the second equation
\begin{align*}
2t_2 +1 +j +i &\equiv 2t_2 +1 \pmod{n}\\
j+i &\equiv  0.
\end{align*}
But, this is not possible because $i,j \in \{ 1, 3,..., \frac{n}{2}-2 \}$.

Therefore, all the pairs of $P_E \cup P_O$ are distinct.

Also, it is obvious for each $p_i \in P_E \cup P_O$,
$$S = \{x_1, x_2 \mid (x_1, x_2) \in p_i\},$$ 
since each $x \in S$ occurs only once in each $p_i$ and each $p_i$ has $\frac{n}{2}$ pairs.\\
\\
Since 
$$\mid P_E \cup P_O \mid \, = \frac{n}{2} + \frac{n}{2} -1 =n-1,$$
we have identified all of the parallel classes.

Since each pair included in the parallel classes of $P_E \cup P_O$ is unique, and the number of all pairs is $(n-1)\frac{n}{2}$, then
$$ \{ (a, b) \mid  (a,b) \in p,\text{ for } p \in P_E \cup P_O\} =\{ (a,b) \mid a < b ; \; a, b \in S \} =T.$$

Therefore, $T$ can be partitioned into the $n-1$ subsets of $P_E \cup P_O$, each of size $\frac{n}{2}$, such that 
the pairs of each $p_i \in P_E \cup P_O$ partition $S$.
\end{proof} 

An example of a partition under Lemma \ref{partition_2} is shown in Appendix \ref{App_example_2}.

\begin{coro}\label{partition_4} Let 
\begin{itemize}
	\item $S = \{ 0, 1, ..., n-1\}$, where $2 \mid n$,
	\item $T = \{ (a,b) \mid a < b ; \; a, b \in S \}$,
\end{itemize}  
and $p_1, p_2,..., p_{n-1}$ be partitions of $T$ such that the pairs of each $p_i$ partition $S$ for each $i \in \{1, 2, ..., n-1 \}$.  
Then
	$$T' = \{ (a,b) \mid a < b ; \; a, b \in S' = \{ 0, 1, ..., 2n-1\} \}$$
can be partitioned into $n$-subsets:  $p'_1, p'_2,..., p'_{2n-1}$, such that the pairs of each $p'_i$ partition $S'$ for each $i \in \{1, 2, ..., 2n-1 \}$.
\end{coro}

\begin{proof}
Let $G_1$ be a complete graph with $n$ vertices labeled $0, 1, ..., n-1$.  The pairs of each $p_i$, for $1 \leq i \leq n-1$, correspond to the edges of $G_1$.  Let $G_2$ be a complete graph, with vertices labeled $n, n+1, ..., 2n-1$, isomorphic to $G_1$ such that each vertex $v$ of $G_1$ corresponds to vertex $v+n$ of $G_2$.  Let $p^*_i$, for $1 \leq i \leq n-1$, be the parallel classes for $G_2$ such that
$$p^*_i = \{ (a+n, b+n) \mid (a,b) \in p_i \} \text{ for } 1 \leq i \leq n-1.$$
\noindent Define the parallel classes for $S'$ as follows, which we group into $P_1$ and $P_2$:
\begin{itemize}
\item $P_1$:  $p'_i = p_i \cup p^*_i$ for $1 \leq i \leq n-1$.\\
	Since each $p'_i$ is a union of a parallel class from $G_1$ and a parallel class from $G_2$, it is obvious that the pairs of each $p'_{i}$ partition $S'$.
\item $P_2$:  $p'_{n+t} = \{ (j, n + (j+t \pmod{n}) ) \mid j = 0, 1,..., n-1 \}$ for $ t \in \{0, 1, ..., n-1 \}.$\\
	For these pairs, the first coordinate comes from $G_1$ and the second coordinate from $G_2$.  Each pair in $p'_{n+t}$ has a second coordinate that is $t$ larger than the $G_2$ vertex that corresponds to the first coordinate.  Therefore, each pair is unique, and the pairs of each $p'_{n+t}$ partition $S'$.
\end{itemize}
Then, 
$$| P_1 | + | P_2 | = n-1 + n = 2n-1.$$
Therefore, we have identified $2n-1$ parallel classes, each of size $n$.  Therefore, $T'$ can be partitioned into $n$-subsets:  $p'_1, p'_2,..., p'_{2n-1}$, such that the pairs of each $p'_i$ partition $S'$ for each $i \in \{1, 2, ..., m-1 \}$.
\end{proof}

An example of a partition under Corollary \ref{partition_4} is shown in Appendix \ref{App_example_4}.

\begin{coro}\label{partition_2_4} Let 
\begin{itemize}
	\item $S'' = \{ 0, 1, ..., m-1\}$, where $2 \mid m$,
	\item $T'' = \{ (a,b) \mid a < b ; \; a, b \in S'' \}$.
\end{itemize}  
Then, $T''$ can be partitioned into $\frac{m}{2}$-subsets:  $p'_1, p'_2,..., p'_{m-1}$, such that the pairs of each $p'_i$ partition $S''$ for each $i \in \{1, 2, ..., m-1 \}$.
\end{coro}
\begin{proof}
This is a proof by mathematical induction.\\
\\
Let 
$$2^s \cdot q =m,$$
for $q$ odd and $s \geq 1$.
\\
\underline{Base Case}.  For $s=1$, let  
\begin{itemize}
\item 
	$m_1 = 2q$,
\item 
	$S_1 = \{ 0, 1, ..., m_1 - 1\}$, 
\item 
	$T_1 = \{ (a,b) \mid a < b ; \; a, b \in S_1 \}$.
\end{itemize}
By Lemma \ref{partition_2}, we know that $T_1$ can be partitioned into $\frac{m_1}{2}$-subsets:  
$$p_{1,1}, p_{1,2}, ... , p_{1, m_1-1},$$ such that the pairs of each $p_{1,i}$ partition $S_1$ for each $i \in \{1, 2, ..., m_1-1 \}$.\\
\\
\underline{Inductive Case}.  For $s \geq 1$, assume \begin{itemize}
\item 
	$m_s = 2^s \cdot q$,
\item 
	$S_s = \{ 0, 1, ..., m_s - 1\}$, 
\item 
	$T_s = \{ (a,b) \mid a < b ; \; a, b \in S_s \}$, and
\item 
	 $T_s$ can be partitioned into $\frac{m_s}{2}$-subsets:  $p_{s,1}, p_{s,2},..., p_{s, m_s-1}$, such that the pairs of each $p_{s,i}$ partition $S_s$ for each $i \in \{1, 2, ..., m_s-1 \}$.
\end{itemize}
Then, based on Corollary \ref{partition_4}, for 
\begin{itemize}
\item 
	$m_{s+1} = 2^{s+1} \cdot q$,
\item 
	$S_{s+1} = \{ 0, 1, ..., m_{s+1} - 1\}$, 
\item 
	$T_{s+1} = \{ (a,b) \mid a < b ; \; a, b \in S_{s+1} \}$,
\end{itemize}
$T_{s+1}$ can be partitioned into $\frac{m_{s+1}}{2}$-subsets:  $p_{{s+1},1}, p_{{s+1},2},..., p_{{s+1}, m_{s+1}-1}$, such that  the pairs of each $p_{s+1,i}$ partition $S_{s+1}$ for each $i \in \{1, 2, ..., m_{s+1}-1 \}$.

Therefore, by mathematical induction, $T''$ can be partitioned into $\frac{m}{2}$-subsets:  $p'_1, p'_2,..., p'_{m-1}$, such that the pairs of each $p'_i$ partition $S''$ for each $i \in \{1, 2, ..., m-1 \}$.
\end{proof}

The following proves that there are an infinite number of Kirkman Quadruple Systems that reach the upper bound for the min-sum.  First, we identify some of the elements that will be used in the proof.

Let $(\mathcal{S}, \mathcal{B})$ be $KQS(n)$ where 
\begin{itemize}
	\item
		$\mathcal{S} = \{0, 1, ..., n-1)\}, $
	\item 
		$\mathcal{B} = \{ B_i \mid i = 1, 2, ..., \frac{n(n-1)(n-2)}{24} \}$ ,  
	\item 
		$\pi_t = \{ (a_{t,i}, b_{t,i}, c_{t,i}, d_{t,i}) \mid 1 \leq i \leq \frac{n}{4}  \}$, for $t = 1, 2,... \frac{(n-1)(n-2)}{6} ,$
	\item 
		with $min_\Sigma(\mathcal{B}) = n+2$.
\end{itemize}
Let $\mathcal{S'} = \{0, 1, ..., 2n-1 \}$.  We will show that the following form the parallel classes ($P_1 \cup P_2$) for  $(\mathcal{S'}, \mathcal{B'})$, a $KQS(2n)$ with $min_\Sigma = 2n+2$.
Let $P_1$ be the union of the following:
\begin{align*}
	\rho_{t,1} &= \{ (a_{t,i}+n, b_{t,i} \qquad , c_{t,i} \qquad , d_{t,i} \qquad  ) , \\
	& \qquad (a_{t,i} \qquad , b_{t,i}+n, c_{t,i}+n, d_{t,i}+n )  \mid 1 \leq i \leq \tfrac{n}{4}  \}, \\
	\rho_{t,2} &= \{ (a_{t,i} \qquad , b_{t,i}+n, c_{t,i} \qquad , d_{t,i} \qquad ) , \\
	& \qquad (a_{t,i}+n, b_{t,i} \qquad , c_{t,i}+n, d_{t,i}+n )  \mid 1 \leq i \leq \tfrac{n}{4}  \}, \\
	\rho_{t,3} &= \{ (a_{t,i} \qquad , b_{t,i} \qquad , c_{t,i}+n, d_{t,i} \qquad ) , \\
	&\qquad (a_{t,i}+n, b_{t,i}+n, c_{t,i} \qquad , d_{t,i}+n )  \mid 1 \leq i \leq \tfrac{n}{4}  \}, \\
	\rho_{t,4} &= \{ (a_{t,i} \qquad , b_{t,i} \qquad , c_{t,i} \qquad , d_{t,i}+n) , \\
	&\qquad (a_{t,i}+n, b_{t,i}+n, c_{t,i}+n, d_{t,i} \qquad  ) \mid 1 \leq i \leq \tfrac{n}{4}  \}, 
\end{align*}  
for $t = 1, 2,..., \frac{(n-1)(n-2)}{6}$.

Let $T = \{ (s,t) \mid s<t; \; s, t \in \mathcal{S} \}$.
Using Corollary \ref{partition_2_4}, partition $T$ into $\frac{n}{2}$-subsets: $p_1, p_2, ..., p_{n-1}$ such that the pairs of each $p_i$ partition $S$ for each $i \in \{1, 2, ..., n-1\}$.

Let $P_2$ be the union of 	
	$$\rho_i = \{ (s, t, s+n, t+n) \mid (s,t) \in p_i \}  \text{, for } i \in \{1, 2, ..., n-1\}.$$

\begin{thm}\label{thm_kqs}
For any $KQS(n)$ with $min_\Sigma = n+2$, a $KQS(2n)$ with $min_\Sigma = 2n+2$ can be constructed.
\end{thm}

\begin{proof}
Let $(\mathcal{S}, \mathcal{B})$ be the $KQS(n)$ described above with $min_\Sigma(\mathcal{B}) = n+2$.  The following will show that $P_1 \cup P_2$ form the parallel classes for $(\mathcal{S'}, \mathcal{B'})$, a $KQS(2n)$ with $min_\Sigma(\mathcal{B'}) = 2n+2$.  
We need to prove the following:
\begin{enumerate}
\item 
	Each triple is not included more than once in $\mathcal{B'}$.
	\begin{itemize}
	\item
		For $P_1$, since each triple of $\mathcal{S}$ is included only once in $\mathcal{B}$, then the blocks of $P_1$ do not contain any duplicate triples.
	\item
		For $P_2$, since each pair is included only once in $T$, there will be no duplicate triples in $P_2$.
	\item
		$P_2$ includes values $s$ and $s+n$, but $P_1$ does not include any such pairs.  Therefore, there are no triples in $P_1$ that are also in $P_2$.
	\end{itemize}
	Therefore, each triple is not included more than once in $\mathcal{B'}$.
\item	
	$| \mathcal{B'} | = \frac{2n(2n-1)(2n-2)}{24} = \frac{n(2n-1)(2n-2)}{12}.$ 
	\begin{align*}
		| \mathcal{B'} | &= \text{Number of blocks in } P_1 \cup P_2 \\
		&= 2 \times \frac{n}{4} \times 4 \times \frac{(n-1)(n-2)}{6} +  \frac{n}{2} \times (n-1)\\
		&= \frac{2n(n-1)(n-2)}{6} + \frac{n(n-1)}{2}  \\
		&= \frac{n(2n-2)(2n-4)}{12} + \frac{n(2n-2) \cdot 3}{12}  \\
		&= \frac{n(2n-2)(2n-1)}{12}.
	\end{align*}
\item
	Each parallel class includes all elements of $\mathcal{S'}$.
	\begin{itemize}
	\item
		Since parallel class $\rho_{t, i}$ of $P_1$ contains 
		\begin{align*}
		a_{t,i} &, \; a_{t,i}+n,\\
		b_{t,i} &, \; b_{t,i}+n,\\
		c_{t,i} &, \; c_{t,i}+n,\\
		d_{t,i} &, \; d_{t,i}+n,\\
		\end{align*}
		for some parallel class $\pi_t$ of $(\mathcal{S}, \mathcal{B})$, then $\rho_{t, i}$ must contain all the elements of $\mathcal{S'}$.
	\item
		For $P_2$, since 
		$$\mathcal{S} = \{ s, t \mid (s,t) \in p_i \},$$	
		then 
		$$\rho_i = \{ (s, t, s+n, t+n) \mid (s,t) \in p_i \}$$
		must contain all the elements of $\mathcal{S'}$.
	\end{itemize}
\item	
	$| P_1 \cup P_2 | = \frac{(2n-1)(2n-2)}{6}.$ 
	\begin{align*}
	| P_1 \cup P_2 | &=  4 \times \frac{(n-1)(n-2)}{6} +  (n-1)\\
	&=  \frac{(2n-2)(2n-4)}{6} +  \frac{3(2n-2)}{6}\\
	&=  \frac{(2n-2)(2n-1)}{6}. 
	\end{align*}
\item	
	$min_\Sigma(\mathcal{B'}) = 2n+2.$
	\begin{itemize}
	\item
		Since 
		$$min_\Sigma(\mathcal{B}) = n+2,$$
		and each block of $P_1$ is a block of $\mathcal{B}$ increased by at least $n$, then 
		$$min_\Sigma(P_1) \geq 2n+2.$$
	\item
		For $P_2$, since $s$ or $t$ is at least 1, 
		$$min_\Sigma(P_2) = s + t +s+n + t+n \geq 2n+2.$$
	\end{itemize}
\end{enumerate}
Therefore, $(\mathcal{S'}, \mathcal{B'})$ is $KQS(2n)$ with $min_\Sigma(\mathcal{B'}) = 2n+2$.
\end{proof}

\begin{coro}\label{coro_kqs}
For every $k \geq 0$, there exists a $KQS(n= 4 \cdot 2^k)$ that reaches the upper bound of the min-sum.
\end{coro}
\begin{proof}
This is a proof by mathematical induction.

\underline{Base case}.  For $k=0$, let $n=4 \cdot 2^k$.  Then
\begin{align*}
\mathcal{S}_0 &= \{0, 1, 2, 3 \}\\
\mathcal{B}_0 &= \{ (0, 1, 2, 3) \}\\
\pi_1 &= \{ (0, 1, 2, 3) \}\\
min_\Sigma(\mathcal{B}_0) &= 6 = n+2.
\end{align*}
Therefore, there exists a $KQS(n)$ with $min_\Sigma = n$ for $k=0$.

\underline{Inductive Case}.  Assume that there exists $k \geq 0$, such that for $n=4 \cdot 2^k$, there exists $(\mathcal{S}_k, \mathcal{B}_k)$, a $KQS(n)$ where $\mathcal{S}_k = \{0, 1, ..., n-1 \}$ with $min_\Sigma(\mathcal{B}_k) = n+2$.  Then, based on Theorem \ref{thm_kqs}, there exists $(\mathcal{S}_{k+1}, \mathcal{B}_{k+1})$, a $KQS(2n)$ where $\mathcal{S}_{k+1} = \{0, 1, ..., 2n-1 \}$ with $min_\Sigma(\mathcal{B}_{k+1}) = 2n+2$. 

Therefore, there exists $KQS(n = 4 \cdot 2^k)$ for all $k \geq 0$, that reaches the upper bound of min-sum.
\end{proof}

\section{Conclusion}
Theorem \ref{theorem_KTS(3n)} and its corollary prove that for every $k \geq 1$, there exists a $KTS(n = 3^k)$ with a minimum block sum of $n$, which is the upper bound for this value.

Theorem \ref{thm_kqs} and its corollary prove that for every $k \geq 0$, there exists a $KQS(n = 4 \cdot 2^k)$ with a minimum block sum of $n+2$, which is the upper bound for this value.

These conclusions can then be used to distribute servers by location, as described in Section 1, so that each location has all the data, while the chunks of data have been distributed across servers by popularity.

\begin{appendices}

\section{Example for Theorem \ref{theorem_KTS(3n)}}\label{App_example_3n}
The following provides an example of using a $KTS(9)$ with $\min_\Sigma =9$ to create a $KTS(27)$ with $\min_\Sigma = 27$.\\
\\
The parallel classes for KTS(9) are $\pi_i$ for i=1, 2, 3, 4.\\
$
\pi  _{1} = \{( 0, 1, 8), ( 3, 4, 2), ( 6, 7, 5)\} \\
\pi  _{2} = \{( 0, 2, 7), ( 3, 5, 1), ( 6, 8, 4)\} \\
\pi  _{3} = \{( 1, 2, 6), ( 4, 5, 0), ( 7, 8, 3)\} \\
\pi  _{4} = \{( 0, 3, 6), ( 1, 4, 7), ( 2, 5, 8)\} \\
$
\\
Using  Theorem \ref{theorem_KTS(3n)} results in the following parallel classes for order 27:\\
\begin{align*}
\pi ' _{0} &= \{( 0, 9, 18), ( 1, 10, 19), ( 2, 11, 20), ( 3, 12, 21), ( 4, 13, 22),\\ &\qquad ( 5, 14, 23), ( 6, 15, 24), ( 7, 16, 25), ( 8, 17, 26)\} \\
\pi ' _{1, 1} &= \{( 0, 1, 26), ( 9, 10, 8), ( 18, 19, 17), ( 3, 4, 20), ( 12, 13, 2),\\ &\qquad ( 21, 22, 11), ( 6, 7, 23), ( 15, 16, 5), ( 24, 25, 14)\} \\
\pi ' _{2, 1} &= \{( 0, 8, 19), ( 9, 17, 1), ( 18, 26, 10), ( 3, 2, 22), ( 12, 11, 4),\\ &\qquad ( 21, 20, 13), ( 6, 5, 25), ( 15, 14, 7), ( 24, 23, 16)\} \\
\pi ' _{3, 1} &= \{( 1, 8, 18), ( 10, 17, 0), ( 19, 26, 9), ( 4, 2, 21), ( 13, 11, 3),\\ &\qquad ( 22, 20, 12), ( 7, 5, 24), ( 16, 14, 6), ( 25, 23, 15)\} 
\end{align*}
\begin{align*}
\pi ' _{1, 2} &= \{( 0, 2, 25), ( 9, 11, 7), ( 18, 20, 16), ( 3, 5, 19), ( 12, 14, 1),\\ &\qquad ( 21, 23, 10), ( 6, 8, 22), ( 15, 17, 4), ( 24, 26, 13)\} \\
\pi ' _{2, 2} &= \{( 0, 7, 20), ( 9, 16, 2), ( 18, 25, 11), ( 3, 1, 23), ( 12, 10, 5),\\ &\qquad ( 21, 19, 14), ( 6, 4, 26), ( 15, 13, 8), ( 24, 22, 17)\} \\
\pi ' _{3, 2} &= \{( 2, 7, 18), ( 11, 16, 0), ( 20, 25, 9), ( 5, 1, 21), ( 14, 10, 3),\\ &\qquad ( 23, 19, 12), ( 8, 4, 24), ( 17, 13, 6), ( 26, 22, 15)\} 
\end{align*}
\begin{align*}
\pi ' _{1, 3} &= \{( 1, 2, 24), ( 10, 11, 6), ( 19, 20, 15), ( 4, 5, 18), ( 13, 14, 0),\\ &\qquad ( 22, 23, 9), ( 7, 8, 21), ( 16, 17, 3), ( 25, 26, 12)\} \\
\pi ' _{2, 3} &= \{( 1, 6, 20), ( 10, 15, 2), ( 19, 24, 11), ( 4, 0, 23), ( 13, 9, 5),\\ &\qquad ( 22, 18, 14), ( 7, 3, 26), ( 16, 12, 8), ( 25, 21, 17)\} \\
\pi ' _{3, 3} &= \{( 2, 6, 19), ( 11, 15, 1), ( 20, 24, 10), ( 5, 0, 22), ( 14, 9, 4),\\ &\qquad ( 23, 18, 13), ( 8, 3, 25), ( 17, 12, 7), ( 26, 21, 16)\} 
\end{align*}
\begin{align*}
\pi ' _{1, 4} &= \{( 0, 3, 24), ( 9, 12, 6), ( 18, 21, 15), ( 1, 4, 25), ( 10, 13, 7),\\ &\qquad ( 19, 22, 16), ( 2, 5, 26), ( 11, 14, 8), ( 20, 23, 17)\} \\
\pi ' _{2, 4} &= \{( 0, 6, 21), ( 9, 15, 3), ( 18, 24, 12), ( 1, 7, 22), ( 10, 16, 4),\\ &\qquad ( 19, 25, 13), ( 2, 8, 23), ( 11, 17, 5), ( 20, 26, 14)\} \\
\pi ' _{3, 4} &= \{( 3, 6, 18), ( 12, 15, 0), ( 21, 24, 9), ( 4, 7, 19), ( 13, 16, 1),\\ &\qquad ( 22, 25, 10), ( 5, 8, 20), ( 14, 17, 2), ( 23, 26, 11)\} 
\end{align*}
These blocks have $\min_\Sigma = 27$. 

\section{Example for Lemma \ref{partition_2}}\label{App_example_2}
The following provides an example of the partition of pairs under Lemma \ref{partition_2} for $n=6$.  Using the methodology of that lemma, we have the following parallel classes:
\begin{align*}
p_{E0} &= \{ (0, 3), (1, 5), (2, 4) \}\\
p_{E1} &= \{ (1, 4), (2, 0), (3, 5) \}\\
p_{E2} &= \{ (2, 5), (3, 1), (4, 0) \}\\
p_{O1} &= \{ (2, 3), (4, 5), (0, 1) \}\\
p_{O2} &= \{ (3, 4), (5, 0), (1, 2) \}
\end{align*}

\section{Example for Corollary \ref{partition_4}}\label{App_example_4}
The following provides an example of the partition of pairs under Corollary \ref{partition_4} for $2n=12$.  Using the methodology of that corollary, we will start with the following parallel classes for $n=6$ corresponding to $G_1$, as derived in Appendix \ref{App_example_2}:
\begin{align*}
p_{1} &= \{ (0, 3), (1, 5), (2, 4) \}\\
p_{2} &= \{ (1, 4), (2, 0), (3, 5) \}\\
p_{3} &= \{ (2, 5), (3, 1), (4, 0) \}\\
p_{4} &= \{ (2, 3), (4, 5), (0, 1) \}\\
p_{5} &= \{ (3, 4), (5, 0), (1, 2) \}
\end{align*}
The parallel classes corresponding to $G_2$ are as follows:
\begin{align*}
p^*_{1} &= \{ (6, 9), (7, 11), (8, 10) \}\\
p^*_{2} &= \{ (7, 10), (8, 6), (9, 11) \}\\
p^*_{3} &= \{ (8, 11), (9, 7), (10, 6) \}\\
p^*_{4} &= \{ (8, 9), (10, 11), (6, 7) \}\\
p^*_{5} &= \{ (9, 10), (11, 6), (7, 8) \}
\end{align*}
Then the parallel classes for $P_1$ are as follows:
\begin{align*}
p'_{1} &= \{ (0, 3), (1, 5), (2, 4), (6, 9), (7, 11), (8, 10) \}\\
p'_{2} &= \{ (1, 4), (2, 0), (3, 5), (7, 10), (8, 6), (9, 11) \}\\
p'_{3} &= \{ (2, 5), (3, 1), (4, 0), (8, 11), (9, 7), (10, 6) \}\\
p'_{4} &= \{ (2, 3), (4, 5), (0, 1), (8, 9), (10, 11), (6, 7) \}\\
p'_{5} &= \{ (3, 4), (5, 0), (1, 2), (9, 10), (11, 6), (7, 8) \}
\end{align*}
The parallel classes for $P_2$ are as follows:
\begin{align*}
p'_{6} &= \{ (0, 6), (1, 7), (2, 8), (3, 9), (4, 10), (5, 11) \}\\
p'_{7} &= \{ (0, 7), (1, 8), (2, 9), (3, 10), (4, 11), (5, 6) \}\\
p'_{8} &= \{ (0, 8), (1, 9), (2, 10), (3, 11), (4, 6), (5, 7) \}\\
p'_{9} &= \{ (0, 9), (1, 10), (2, 11), (3, 6), (4, 7), (5, 8) \}\\
p'_{10} &= \{ (0, 10), (1, 11), (2, 6), (3, 7), (4, 8), (5, 9) \}\\
p'_{11} &= \{ (0, 11), (1, 6), (2, 7), (3, 8), (4, 9), (5, 10) \}
\end{align*}

\section{Example for Theorem \ref{thm_kqs}}\label{App_example_kqs}
The following provides an example of using a $KQS(8)$ with $\min_\Sigma =10$ to create a $KQS(16)$ with $\min_\Sigma = 18$.\\
\\
The parallel classes for $KQS(8)$ are the following $\pi_i$ for $i = 1,2,...7$:\\
$
\pi  _{1} = \{( 4, 1, 2, 3), ( 0, 5, 6, 7)\} \\
\pi  _{2} = \{( 0, 5, 2, 3), ( 4, 1, 6, 7)\} \\
\pi  _{3} = \{( 0, 1, 6, 3), ( 4, 5, 2, 7)\} \\
\pi  _{4} = \{( 0, 1, 2, 7), ( 4, 5, 6, 3)\} \\
\pi  _{5} = \{( 0, 1, 4, 5), ( 2, 3, 6, 7)\} \\
\pi  _{6} = \{( 0, 2, 4, 6), ( 1, 3, 5, 7)\} \\
\pi  _{7} = \{( 0, 3, 4, 7), ( 1, 2, 5, 6)\} \\
$
\\
The following partitions of $\{ 0, 1, ..., 7 \}$ are used:\\
$
p _{1} = \{( 0, 1), ( 2, 3), ( 4, 5), ( 6, 7)\} \\
p _{2} = \{( 0, 2), ( 1, 3), ( 4, 6), ( 5, 7)\} \\
p _{3} = \{( 0, 3), ( 1, 2), ( 4, 7), ( 5, 6)\} \\
p _{4} = \{( 0, 4), ( 1, 5), ( 2, 6), ( 3, 7)\} \\
p _{5} = \{( 0, 5), ( 1, 6), ( 2, 7), ( 3, 4)\} \\
p _{6} = \{( 0, 6), ( 1, 7), ( 2, 4), ( 3, 5)\} \\
p _{7} = \{( 0, 7), ( 1, 4), ( 2, 5), ( 3, 6)\} \\
$
\\
Using  Theorem \ref{thm_kqs} results in the following parallel classes for order 16:\\
$
\rho  _{1,1} = \{ 12, 1, 2, 3), ( 4, 9, 10, 11), ( 8, 5, 6, 7), ( 0, 13, 14, 15)\} \\
\rho  _{1,2} = \{ 4, 9, 2, 3), ( 12, 1, 10, 11), ( 0, 13, 6, 7), ( 8, 5, 14, 15)\} \\
\rho  _{1,3} = \{ 4, 1, 10, 3), ( 12, 9, 2, 11), ( 0, 5, 14, 7), ( 8, 13, 6, 15)\} \\
\rho  _{1,4} = \{ 4, 1, 2, 11), ( 12, 9, 10, 3), ( 0, 5, 6, 15), ( 8, 13, 14, 7)\} \\
\\
\rho  _{2,1} = \{ 8, 5, 2, 3), ( 0, 13, 10, 11), ( 12, 1, 6, 7), ( 4, 9, 14, 15)\} \\
\rho  _{2,2} = \{ 0, 13, 2, 3), ( 8, 5, 10, 11), ( 4, 9, 6, 7), ( 12, 1, 14, 15)\} \\
\rho  _{2,3} = \{ 0, 5, 10, 3), ( 8, 13, 2, 11), ( 4, 1, 14, 7), ( 12, 9, 6, 15)\} \\
\rho  _{2,4} = \{ 0, 5, 2, 11), ( 8, 13, 10, 3), ( 4, 1, 6, 15), ( 12, 9, 14, 7)\} \\
\\
\rho  _{3,1} = \{ 8, 1, 6, 3), ( 0, 9, 14, 11), ( 12, 5, 2, 7), ( 4, 13, 10, 15)\} \\
\rho  _{3,2} = \{ 0, 9, 6, 3), ( 8, 1, 14, 11), ( 4, 13, 2, 7), ( 12, 5, 10, 15)\} \\
\rho  _{3,3} = \{ 0, 1, 14, 3), ( 8, 9, 6, 11), ( 4, 5, 10, 7), ( 12, 13, 2, 15)\} \\
\rho  _{3,4} = \{ 0, 1, 6, 11), ( 8, 9, 14, 3), ( 4, 5, 2, 15), ( 12, 13, 10, 7)\} \\
\\
\rho  _{4,1} = \{ 8, 1, 2, 7), ( 0, 9, 10, 15), ( 12, 5, 6, 3), ( 4, 13, 14, 11)\} \\
\rho  _{4,2} = \{ 0, 9, 2, 7), ( 8, 1, 10, 15), ( 4, 13, 6, 3), ( 12, 5, 14, 11)\} \\
\rho  _{4,3} = \{ 0, 1, 10, 7), ( 8, 9, 2, 15), ( 4, 5, 14, 3), ( 12, 13, 6, 11)\} \\
\rho  _{4,4} = \{ 0, 1, 2, 15), ( 8, 9, 10, 7), ( 4, 5, 6, 11), ( 12, 13, 14, 3)\} \\
\\
\rho  _{5,1} = \{ 8, 1, 4, 5), ( 0, 9, 12, 13), ( 10, 3, 6, 7), ( 2, 11, 14, 15)\} \\
\rho  _{5,2} = \{ 0, 9, 4, 5), ( 8, 1, 12, 13), ( 2, 11, 6, 7), ( 10, 3, 14, 15)\} \\
\rho  _{5,3} = \{ 0, 1, 12, 5), ( 8, 9, 4, 13), ( 2, 3, 14, 7), ( 10, 11, 6, 15)\} \\
\rho  _{5,4} = \{ 0, 1, 4, 13), ( 8, 9, 12, 5), ( 2, 3, 6, 15), ( 10, 11, 14, 7)\} \\
\\
\rho  _{6,1} = \{ 8, 2, 4, 6), ( 0, 10, 12, 14), ( 9, 3, 5, 7), ( 1, 11, 13, 15)\} \\
\rho  _{6,2} = \{ 0, 10, 4, 6), ( 8, 2, 12, 14), ( 1, 11, 5, 7), ( 9, 3, 13, 15)\} \\
\rho  _{6,3} = \{ 0, 2, 12, 6), ( 8, 10, 4, 14), ( 1, 3, 13, 7), ( 9, 11, 5, 15)\} \\
\rho  _{6,4} = \{ 0, 2, 4, 14), ( 8, 10, 12, 6), ( 1, 3, 5, 15), ( 9, 11, 13, 7)\} \\
\\
\rho  _{7,1} = \{ 8, 3, 4, 7), ( 0, 11, 12, 15), ( 9, 2, 5, 6), ( 1, 10, 13, 14)\} \\
\rho  _{7,2} = \{ 0, 11, 4, 7), ( 8, 3, 12, 15), ( 1, 10, 5, 6), ( 9, 2, 13, 14)\} \\
\rho  _{7,3} = \{ 0, 3, 12, 7), ( 8, 11, 4, 15), ( 1, 2, 13, 6), ( 9, 10, 5, 14)\} \\
\rho  _{7,4} = \{ 0, 3, 4, 15), ( 8, 11, 12, 7), ( 1, 2, 5, 14), ( 9, 10, 13, 6)\} \\
\\
$
$
\rho  _{1} = \{ 0, 1, 8, 9), ( 2, 3, 10, 11), ( 4, 5, 12, 13), ( 6, 7, 14, 15)\} \\
\rho  _{2} = \{ 0, 2, 8, 10), ( 1, 3, 9, 11), ( 4, 6, 12, 14), ( 5, 7, 13, 15)\} \\
\rho  _{3} = \{ 0, 3, 8, 11), ( 1, 2, 9, 10), ( 4, 7, 12, 15), ( 5, 6, 13, 14)\} \\
\rho  _{4} = \{ 0, 4, 8, 12), ( 1, 5, 9, 13), ( 2, 6, 10, 14), ( 3, 7, 11, 15)\} \\
\rho  _{5} = \{ 0, 5, 8, 13), ( 1, 6, 9, 14), ( 2, 7, 10, 15), ( 3, 4, 11, 12)\} \\
\rho  _{6} = \{ 0, 6, 8, 14), ( 1, 7, 9, 15), ( 2, 4, 10, 12), ( 3, 5, 11, 13)\} \\
\rho  _{7} = \{ 0, 7, 8, 15), ( 1, 4, 9, 12), ( 2, 5, 10, 13), ( 3, 6, 11, 14)\} \\
$\\
These blocks have $\min_\Sigma = 18$.

\end{appendices}



\end{document}